\documentclass[12pt]{amsart}
\usepackage{amssymb, hyperref, mathrsfs,xcolor, cancel, enumerate, comment, fullpage}

\newtheorem{questionIntro}{Question}

\newtheorem{theorem}{Theorem}[section]
\newtheorem{lemma}[theorem]{Lemma}
\newtheorem{proposition}[theorem]{Proposition}

\newtheorem{definition}[theorem]{Definition}
\newtheorem{corollary}[theorem]{Corollary}

\theoremstyle{remark}
\newtheorem{remark}[theorem]{Remark}

\newcommand{\calP}{\ensuremath{\mathcal{P}}}
\newcommand{\calS}{\ensuremath{\mathcal{S}}}

\newcommand{\calW}{\ensuremath{\mathcal{W}}}
\newcommand{\calX}{\ensuremath{\mathcal{X}}}

\newcommand{\calA}{\ensuremath{\mathcal{A}}}

\newcommand{\calC}{\ensuremath{\mathcal{C}}}

\newcommand{\mo}{{-1}}

\newcommand{\bbZ}{\ensuremath{\mathbb{Z}}}

\newcommand{\bbQ}{\ensuremath{\mathbb{Q}}}

\newcommand{\cyc}{\ensuremath{\mathrm{cyc}}}
\newcommand{\ab}{\ensuremath{\mathrm{ab}}}
\newcommand{\sol}{\ensuremath{\mathrm{sol}}}
\newcommand{\ssol}{\ensuremath{\mathrm{ssol}}}
\newcommand{\nil}{\ensuremath{\mathrm{nil}}}
\begin{document}

\title{Asymptotic behaviour of minimal complements }
 
\author{Arindam Biswas}
\address{Department of Mathematics, Technion - Israel Institute of Technology, Haifa 32000, Israel}
\curraddr{}
\email{biswas@campus.technion.ac.il}
\thanks{}

\author{Jyoti Prakash Saha}
\address{Department of Mathematics, Indian Institute of Science Education and Research Bhopal, Bhopal Bypass Road, Bhauri, Bhopal 462066, Madhya Pradesh,
India}
\curraddr{}
\email{jpsaha@iiserb.ac.in}
\thanks{}



\begin{abstract}
The notion of minimal complements was introduced by Nathanson in 2011 as a natural group-theoretic analogue of the metric concept of nets. Given two non-empty subsets $W,W'$ in a group $G$, the set $W'$ is said to be a complement to $W$ if $W\cdot W'=G$ and it is minimal if no proper subset of $W'$ is a complement to $W$. The inverse problem asks which sets may or not occur as minimal complements. We show some new results on the inverse problem and investigate how the study of the inverse problem naturally gives rise to questions about the asymptotic behaviour of these sets, providing partial answers to some of them. 
\end{abstract}

\maketitle

 
\section{Introduction}
\subsection{Motivation}
 Let $A, B$ be non-empty subsets in a group $G$. The set $A$ is said to be a left (resp. right) complement to $B$ if $A \cdot B = G$ (resp. $B\cdot A = G$). The set $A$ is a minimal left complement to $B$ if $$A \cdot B = G \textnormal{ and } (A\setminus \lbrace a\rbrace )\cdot B \subsetneq G \text{ for any } a\in A.$$ 
The minimal right complements are defined analogously. In the literature, complements (as defined above) are also known as additive or multiplicative complements (depending on the group structure)  to distinguish them from set-theoretic complements, but in this article we shall use the term complement or minimal complement to mean the above. 
 
 The study of minimal complements began with Nathanson in \cite{NathansonAddNT4}, who introduced the notion in the context of additive number theory and geometric group theory as an analogue of the metric concept of nets adapted to groups. Indeed, (group-theoretic) nets and minimal nets are related with complements and minimal complements. See \cite[Lemma 2]{NathansonAddNT4}. In the same article, he posed certain questions regarding the classification of sets which admit minimal complements. We shall refer to these problems as the direct problems. Works on the direct problems include those of Nathanson \cite{NathansonAddNT4}, Chen--Yang \cite{ChenYang12}, Kiss--S\'{a}ndor--Yang \cite{KissSandorYangJCT19}, the authors \cite{MinComp1}, \cite{MinComp2} etc. Recently, the study of sets which may or may not occur as minimal complements has also become popular. We shall refer to them as the inverse problems (a term coined by Alon--Kravitz--Larson). Works mainly on the inverse problems include those of Kwon \cite{Kwon}, Alon--Kravitz--Larson \cite{AlonKravitzLarson}, Burcroff--Luntzlara \cite{BurcroffLuntzlara}, the authors  \cite{CoMin1,CoMin2, CoMin3, CoMin4} etc. (in fact, \cite{CoMin2}, \cite{CoMin3} deal with co-minimal pairs (see Definition \ref{Def:CoMin}) and hence are concerned with both the direct and the inverse problems). It is often the case that the inverse problems are harder to answer, e.g., even if we restrict to finite groups, it is easy to see that any non-empty subset admits a minimal complement, but the corresponding inverse problem, asking whether any non-empty subset occurs as a minimal complement or not, has a negative answer. This forms the basis of our investigation on the asymptotic behaviour of sets which occur as minimal complements. 

\subsection{Results obtained}
In the context of groups of several kinds, it follows from the works of Alon--Kravitz--Larson (see Proposition \ref{Prop:23rdBdd}), Burcroff--Luntzlara \cite[Lemma 5]{BurcroffLuntzlara} and that of the authors \cite[Theorem C]{CoMin1}, \cite[Corollary 2.9]{CoMin1} (see also \cite{CoMin4})  that ``large" subsets cannot occur as minimal complements. 
On the other hand, the recent works of Kwon \cite[Theorem 9]{Kwon}, Alon--Kravitz--Larson (Theorem \ref{Thm:SmallWorks}) and the authors \cite[Theorem B]{CoMin1} show that the ``small'' subsets of several groups are minimal complements. However, it is not established that any nonempty finite subset of any infinite group is a minimal complement (to the best of knowledge of the authors). Using the ideas of the proof of \cite[Theorem 16]{AlonKravitzLarson}, we prove Theorem \ref{Thm:SmallInNon-abelian}, which implies that it is indeed the case (see Corollary \ref{Cor:SmallInNon-abelian}). 

\begin{theorem}\label{Thm:SmallInNon-abelian}
If $C$ is a nonempty finite subset of a group $G$ such that $|G| > |C|^5 - |C|^4$, then $C$ is a minimal complement in $G$. 
\end{theorem}

\begin{corollary}\label{Cor:SmallInNon-abelian}
Any nonempty finite subset of any infinite group is a minimal complement. 
\end{corollary}

The above corollary generalizes \cite[Theorem 9]{Kwon}, \cite[Theorem B]{CoMin1}, \cite[Theorem 2]{AlonKravitzLarson}. 
Moreover, increasing our understanding of which sets occur as minimal complements, the following result is also shown in section \ref{Sec:MinComp}.

\begin{theorem}\label{Thm:UnionOfCosetsInZ^d}
Let $d, n, k$ be a positive integers such that 
$$k \leq \frac{n^{d/3}} {2(d \log_2 n)^{2/3}}.$$
Let $\calX_1, \cdots, \calX_k$ be subsets of $\bbZ^d$ which are minimal complements in $\bbZ^d$. Let $c_1, \cdots, c_k$ be elements of $\bbZ^d$ which are pairwise distinct modulo $n\bbZ^d$. Then 
$\cup_{1 \leq i \leq k} (c_i + n\calX_i)$ 
is a minimal complement in $\bbZ^d$. 
Moreover, if any nonempty subset of $\bbZ^d$ having finite symmetric difference with any one of $\calX_1, \cdots, \calX_k$ is a minimal complement in $\bbZ^d$, then any nonempty subset of $\cup_{1\leq i \leq k} (c_i + n\bbZ^d)$ having finite symmetric difference with 
$\cup_{1 \leq i \leq k} (c_i + n\calX_i)$ 
is a minimal complement in $\bbZ^d$. 
\end{theorem}

Let $G$ be a finite group of order $n$ and consider the collection $\mathcal{C}$ of all non-empty subsets of $G$ which occur as minimal complements. There are several immediate questions about the elements of  $\mathcal{C}$, e.g., what are the sizes of the elements of $\mathcal{C}$, what are the asymptotic properties of the sizes as $n\rightarrow \infty$, what are the integers $k$ between $1$ and $n$ such that any subset (or some subset) of $G$ of size $k$ is a minimal complement? 
In a prior work of the authors, some such questions were asked \cite[Question 1]{CoMin1}. 
One can also study these questions by restricting to particular classes of groups, e.g., in the context of cyclic groups, or abelian groups, or non-abelian groups. We shall investigate these questions and provide partial answers to some of them in section \ref{Sec:Limit}. In section \ref{Sec:AsyCoMin}, we shall study the above questions for co-minimal pairs (see Definition \ref{Def:CoMin}). 

\section{Background literature} 
\label{Sec:BackLit}
To study the asymptotic behaviour of minimal complements and that of co-minimal pairs we shall repeatedly use some previous results. Most of them are very recent and not yet in widespread use, so it is worthwhile to collect them here.

\begin{theorem}
[{\cite[Theorem 1]{AlonKravitzLarson}}]\label{Thm:SmallWorks}
Let $G$ be a group of order $n\geq 2$. If $C$ is a nonempty subset of $G$ of size 
$$\leq \frac{n^{1/3}}{2(\log_2 n)^{2/3}},$$
then $C$ is a minimal left complement to some subset in $G$ and it is a minimal right complement to some subset in $G$. 
\end{theorem}

\begin{proposition}
[{\cite[Proposition 13]{AlonKravitzLarson}}]
\label{Prop:23rdBdd}
Let $G$ be a finite group. If a subset $C$ of $G$ is a minimal complement to some subset $W$, 
then 
$$|C| 
\leq |G| \frac{|W|} {2|W| -1}.$$
\end{proposition}

Alon, Kravitz and Larson showed the above results in the context of abelian groups, but it can be seen that their proof extends to the setting when $G$ is not assumed to be abelian (by replacing the sums of the form $a+b$ (resp. $a-b$) by $a\cdot b$ (resp. $a\cdot b^\mo$) and the sets of the form $a+B$ (resp. $a-B$) by $a\cdot B$ (resp. $a\cdot B^\mo$). Recently, Burcroff and Luntzlara have proved a result which is more general than 
Proposition \ref{Prop:23rdBdd} in the context of abelian groups 
\cite[Lemma 5]{BurcroffLuntzlara}.

\begin{theorem}
[{\cite[Theorem B]{CoMin1}}]
Given any nonempty subset $S$ of a group $G$ with $|S|\leq 2$, there are subsets $L, R$ of $G$ such that $(S, R), (L, S)$ are co-minimal pairs. 
\end{theorem}

\begin{proposition}
[{\cite[Proposition 2.17]{CoMin4}}]
\label{Prop:fini}
Let $G$ be a finite group and $C$ be a subset of a subgroup $H$ of $G$ satisfying 
$$ |H| > |C| > 2[G:H] |H\setminus C|.$$
Then $C$ is not a minimal complement to any subset of $G$. 
\end{proposition}

Proposition \ref{Prop:fini} was first established by Alon, Kravitz and Larson in the context of abelian groups \cite[Proposition 17]{AlonKravitzLarson}.

\begin{proposition}
\label{Prop:CoMinCartesian}
If $G_1, G_2$ are groups and $(A_1, B_1)$ (resp. $(A_2, B_2)$) is a co-minimal pair in $G_1$ (resp. $G_2$), then $(A_1 \times A_2, B_1\times B_2)$ is a co-minimal pair in $G_1\times G_2$. 
\end{proposition}

The above result is a special case of \cite[Proposition 3.2]{CoMin1}. 

\section{Sets occurring as minimal complements}
\label{Sec:MinComp}

In this section, we exhibit several sets that occur as minimal complements.

\begin{theorem}
\label{Thm:MinCompTwoSidedTra}
Let $C$ be a nonempty subset of a group $G$. Assume that the set $C$ contains a right translate of itself only if it is equal to $C$. If for each $c\in C$, there exists an element $g_c\in G$ such that the sets $g_c \cdot C^\mo \cdot C$ are pairwise disjoint, then $C$ is a minimal right complement in $G$. 
Moreover, if $C$ is finite and 
$$|G| > |C|^5 - |C|^4,$$
then $C$ is a minimal complement in $G$. 
\end{theorem}

\begin{proof}
For $c\in C$, let $w_c$ denote the element $g_c \cdot c^\mo$. It follows that the sets $w_c\cdot c \cdot C^\mo \cdot C$ are pairwise disjoint. 
It also follows that the sets $w_c\cdot c \cdot C^\mo$ are pairwise disjoint. 
Let $W$ denote the union of the sets $\{w_c\,|\,c\in C\}$ and $G \setminus (\cup_{c\in C} w_c \cdot  c \cdot C^\mo)$. 
Choose an element $z$ of $G$. If $G \setminus (\cup_{c\in C} w_c \cdot  c \cdot C^\mo)$ contains some element of $z\cdot C^\mo$, then $W\cdot C$ contains $z$. Suppose $z\cdot C^\mo$ is contained in $\cup_{c\in C} w_c \cdot  c \cdot C^\mo$. 
Since the sets $w_c\cdot c \cdot C^\mo \cdot C$ are pairwise disjoint, it follows that $z\cdot C^\mo$ is contained in $w_c \cdot  c \cdot C^\mo$ for exactly one $c\in C$. By the hypothesis, we obtain $z\cdot C^\mo = w_c\cdot c \cdot C^\mo$. So $z$ belongs to $W\cdot C$. 
It follows that $W\cdot C = G$. Since the sets $w_c\cdot c \cdot C^\mo$ are pairwise disjoint, it follows that $w_c \cdot c \notin w_d \cdot C$ for any two distinct $c, d \in C$. So, $C$ is a minimal right complement to $W$. 

Note that given finite subsets $A_1, \cdots, A_r$ of $G$, there exist elements $g_1, \cdots, g_r$ in $G$ such that $g_1\cdot A_1, \cdots, g_r\cdot A_r$ are pairwise disjoint 
if 
$$|G| > |A_1 \cdot A_s^\mo | + \cdots + |A_{s-1} \cdot A_s^\mo| $$
holds for any $1< s \leq r$. 
If $|G| > |C|^5 - |C|^4$, then there exist $w_1, \cdots, w_k$ in $G$ such that the sets $w_i\cdot c_i \cdot C^\mo \cdot C$ are pairwise disjoint. 
Let $W$ denote the union of the sets $\{w_1, \cdots, w_k\}$ and $G \setminus (\cup_{1\leq i \leq k} w_i \cdot  c_i \cdot C^\mo)$. 
Choose an element $z$ of $G$. 
Since the sets $w_i\cdot c_i \cdot C^\mo \cdot C$ are pairwise disjoint, at most one of $w_1 \cdot c_1 \cdot C^\mo, \cdots, w_k \cdot c_k \cdot C^\mo$ intersects with $z\cdot C^\mo$. 
If $w_i\cdot c_i \cdot C^\mo$ intersects with $z\cdot C^\mo$ for some $i$, 
then $z$ belongs to $W\cdot C$ if $z\cdot C^\mo = w_i\cdot c_i \cdot C^\mo$, 
and $z$ belongs to $W\cdot C$ if $z\cdot C^\mo \neq w_i\cdot c_i \cdot C^\mo$.  
If none of $w_1 \cdot c_1 \cdot C^\mo, \cdots, w_k \cdot c_k \cdot C^\mo$ intersects with $z\cdot C^\mo$, then $z\in W\cdot C$. It follows that $W\cdot C = G$. Since the sets $w_i\cdot c_i \cdot C^\mo \cdot C$ are pairwise disjoint, it follows that the sets $w_i \cdot c_i \cdot C^\mo$ are pairwise disjoint, and hence $w_i \cdot c_i \notin w_j \cdot C$ for any two distinct $i, j$. So, $C$ is a minimal right complement to $W$. 

\end{proof}

\begin{proof}[Proof of Theorem \ref{Thm:SmallInNon-abelian}]
It follows from Theorem \ref{Thm:MinCompTwoSidedTra}. 
\end{proof}

\begin{proof}[Proof of Corollary \ref{Cor:SmallInNon-abelian}]
It follows from Theorem \ref{Thm:MinCompTwoSidedTra}. 
\end{proof}

\begin{corollary}
For any $d\geq 1$, any nonempty bounded subset of $\bbQ^d$ is a minimal complement. 
\end{corollary}

\begin{proof}
It follows from Theorem \ref{Thm:MinCompTwoSidedTra}. 
\end{proof}

One of the crucial steps of the proof of \cite[Theorem 16]{AlonKravitzLarson} is to establish the following. 

\begin{proposition}
\label{Prop:ExistenceOfSKTuple}
Given any finite group $\Gamma$ of order $n$, two integers $s\geq 2, k\geq 1$ satisfying 
$$\frac{s^2 k^3}{n} + \frac{e^s k^{3s}} {n^{s-1}} + k 
\left(
\frac{s^2
k^3}{n}
\right)^s
<1$$
and a subset $C = \{c_1, \cdots, c_k\}$ of $\Gamma$ of size $k$, there exists an $sk$-tuple 
$$
\mathbf w :=
\begin{pmatrix}
w_1^{(1)} & w_1^{(2)} & \cdots & w_1^{(s)}\\
w_2^{(1)} & w_2^{(2)} & \cdots & w_2^{(s)}\\
\vdots & \vdots & \ddots & \vdots \\
w_k^{(1)} & w_k^{(2)} & \cdots & w_k^{(s)}\\
\end{pmatrix}
$$
with values in $\Gamma$ such that each of the following statements is false. 
\begin{enumerate}
\item 
there exist distinct pairs $(i, p), (j, q)$ with $1\leq i, j\leq k, 1\leq p,q\leq s$ such that $w_i^{(p)} \cdot c_i \in w_j^{(q)} \cdot C$. 
\item 
there exist at least $s$ distinct pairs $(i, p)$ such that the corresponding sets $w_i^{(p)} \cdot c_i  \cdot C^\mo \cdot C$ have a nonempty intersection.
\item 
there exists an integer $1\leq i \leq k$ such that for any $1\leq p \leq s$, there exist $z\in \Gamma$, $1\leq j \leq k, j \neq i, 1\leq q\leq s$ such that the following conditions hold. 
\begin{itemize}
\item 
$$
\frac ks < 
|(w_i^{(p)} \cdot c_i \cdot C^\mo) \cap (z\cdot C^\mo)| < k$$
\item 
$w_j^{(q)} \cdot c_j \cdot C^\mo$ contains the first element\footnote{The elements of $z\cdot C^\mo$ are ordered according to the order of the elements of $C$.} of $(z\cdot C^\mo) \setminus 
(w_i^{(p)} \cdot c_i \cdot C^\mo)$. 
\end{itemize}
\end{enumerate}
\end{proposition}

The above result is established by Alon, Kravitz and Larson in the context of abelian groups. Moreover, their argument also works without the hypothesis that the underlying group is abelian and yields the above result.

\begin{theorem}\label{Thm:UnionOfCosets}
Let $G$ be a group and $H$ be a normal subgroup of $G$ of index $n\geq 1$. 
Let $C = \{c_1, \cdots, c_k\} \cdot H$ be a subset of $G$, which is the union of $k$ distinct cosets of $H$ in $G$. Let $\calC_1, \cdots, \calC_k$ be subsets of $c_1 H, \cdots, c_k H$ respectively such that for any $1\leq i \leq k$,  $c_i^\mo \calC_i$ is a minimal right complement in $H$. If there exists an integer $s\geq 2$ satisfying 
$$\frac{s^2 k^3}{n} + \frac{e^s k^{3s}} {n^{s-1}} + k 
\left(
\frac{s^2
k^3}{n}
\right)^s
<1,
$$
then the set 
$$\calC:=\cup_{1\leq i \leq k} \calC_i$$
is a minimal right complement in $G$. 
\end{theorem}

\begin{proof}
For $1\leq i \leq k$, let $\calW_i$ be a subset of $H$ such that $c_i^\mo \cdot \calC_i$ is a minimal right complement to $\calW_i$ in $H$. We will assume that $\calW_i$ contains the identity element. 

In the following, the image of an element $x$ of $G$ under the mod $H$ reduction map $G\to G/H$ is denoted by $\overline x$. 
By Proposition \ref{Prop:ExistenceOfSKTuple}, there exists an $sk$-tuple 
$$
\mathbf w :=
\begin{pmatrix}
w_1^{(1)} & w_1^{(2)} & \cdots & w_1^{(s)}\\
w_2^{(1)} & w_2^{(2)} & \cdots & w_2^{(s)}\\
\vdots & \vdots & \ddots & \vdots \\
w_k^{(1)} & w_k^{(2)} & \cdots & w_k^{(s)}\\
\end{pmatrix}
$$
with values in $G$ such that the following conditions hold. 
\begin{enumerate}[(i)]
\item 
for any two distinct pairs $(i, p), (j, q)$ with $1\leq i, j\leq k, 1\leq p,q\leq s$, $\overline w_i^{(p)} \cdot \overline c_i \notin \overline w_j^{(q)} \cdot \overline C$. 
\item 
the number of pairs $(i, p)$ such that the corresponding sets $\overline w_i^{(p)} \cdot \overline c_i \cdot \overline C^\mo \cdot \overline C$ have a nonempty intersection, is $<s$. 
\item 
for any $1\leq i \leq k$, there exists $1\leq p\leq s$ such that for any $z\in G$ and for any entry of the $sk$-tuple not lying in the $i$-th row (i.e., for any $j\neq i$ and $1\leq q\leq s$), 
at least one of the following conditions is false. 
\begin{itemize}
\item 
$$
\frac ks < 
|(\overline w_i^{(p)} \cdot \overline c_i \cdot \overline C^\mo) \cap (\overline z\cdot \overline C^\mo)| < k$$
\item 
$\overline w_j^{(q)} \cdot \overline c_i \cdot \overline C^\mo$ contains the first element\footnote{The elements of $z\cdot \overline C^\mo$ are ordered according to the order $\overline c_1, \cdots, \overline c_k$ of the elements of $\overline C$.} of $(\overline z\cdot \overline C^\mo) \setminus 
(\overline w_i^{(p)} \cdot \overline c_i \cdot \overline C^\mo)$. 
\end{itemize}
\end{enumerate}
For each $1\leq i \leq k$, choose an integer $1\leq p = p_i \leq s$ as in the third condition above. Let $W$ denote the union of the sets 
$$
w_1^{(p_1)} c_1\calW_1c_1^\mo, \cdots, w_k^{(p_k)} c_k\calW_kc_k^\mo
$$
and $G \setminus (\cup_{1\leq i \leq k} w_i^{(p_i)}\cdot  c_i \cdot C^\mo)$. 
By the first condition, for any $1\leq i \leq k$, no element of 
$w_i^{(p_i)}\cdot c_i H
= 
w_i^{(p_i)}\cdot c_i \calW_i c_i^\mo \calC_i
$
is contained in $w_j^{(p_j)} \calW_j \cdot \calC$ for any $j\neq i$. So, it is enough to show that $W \cdot \calC  = G$ to conclude that $\calC$ is a minimal right complement to $W$.

Choose an element $z$ of $G$. If $z\cdot C^\mo = w_i^{(p_i)} \cdot c_i \cdot C^\mo$ for some $i$, 
then 
$$z\cdot c_{i_z}^\mo \cdot h_z = w_i^{(p_i)}$$
for some $1\leq i_z\leq k, h_z\in H$. This shows that 
$$
z
= w_i^{(p_i)} c_{i_z} c_{i_z}^\mo h_z^\mo c_{i_z}.$$
Since $H$ is normal in $G$, it follows that 
\begin{align*}
z
& \in w_i^{(p_i)} c_{i_z} H \\
& = 
w_i^{(p_i)} c_{i_z} \calW_{i_z} c_{i_z}^\mo \calC_{i_z} \\
& 
\subseteq 
W \cdot \calC.
\end{align*}
So, $z$ lies in $W\cdot \calC$.

Assume that $z\cdot C^\mo \neq w_i^{(p_i)} \cdot c_i \cdot C^\mo$ for any $i$, i.e., none of the $k$ sets $\overline w_1^{(p_1)} \cdot \overline c_1 \cdot \overline C^\mo, \cdots, \overline w_k^{(p_k)} \cdot \overline c_k \cdot \overline C^\mo$ contains $\overline z\cdot \overline C^\mo$. 
By the second condition, at most $s-1$ of them intersect with $\overline z\cdot \overline C^\mo$. If the intersection of each such set with $\overline z\cdot \overline C^\mo$ contains $\leq \frac ks$ elements, then the union of the $k$ sets $\overline w_1^{(p_1)} \cdot \overline c_1 \cdot C^\mo, \cdots, \overline w_k^{(p_k)} \cdot \overline c_k \cdot C^\mo$ does not contain $\overline z\cdot \overline C^\mo$, and hence 
the union of the $k$ sets $w_1^{(p_1)} \cdot c_1 \cdot C^\mo, \cdots, w_k^{(p_k)} \cdot c_k \cdot C^\mo$ does not contain any element of $z\cdot c_r^\mo H$ for some $1\leq r\leq k$. So, the union of the $k$ sets $w_1^{(p_1)} \cdot c_1 \cdot C^\mo, \cdots, w_k^{(p_k)} \cdot c_k \cdot C^\mo$ does not contain any element of $z\cdot \calC_r^\mo$ for some $1\leq r\leq k$. It follows that $z$ lies in $W \cdot \calC$. 

If the intersection of one of those $s-1$ sets with $\overline z\cdot \overline C^\mo$ contains $>\frac ks$ elements, then by the third condition, it follows that for some $i$, the first element of $(\overline z\cdot \overline C^\mo) \setminus 
(\overline w_i^{(p_i)} \cdot \overline c_i \cdot C^\mo)$ is not contained in $\overline w_j^{(q)}\cdot \overline c_j \cdot \overline C^\mo$ for any $j \neq i$ and for any $1\leq q \leq s$, and hence the union of the $k$ sets $\overline w_1^{(p_1)} \cdot \overline c_1 \cdot \overline C^\mo, \cdots, w_k^{(p_k)} \cdot \overline c_k \cdot \overline C^\mo$ does not contain $\overline z\cdot \overline C^\mo$, and consequently, $z$ is contained in $W\cdot \calC$. So $\calC$ is a minimal right complement to $W$ in $G$. 
\end{proof}

\begin{proof}[Proof of Theorem \ref{Thm:UnionOfCosetsInZ^d}]
It follows from Theorem \ref{Thm:UnionOfCosets}. 
\end{proof}

\section{Study of asymptotic behaviour}
\label{Sec:Limit}

For a finite group $G$ of order $n$, let $\calA(G)$ (resp. $\calS(G)$) denote the set of positive integers between $1$ and $n$ such that for any (resp. for some) integer $k\in \calA(G)$ (resp. $\calS(G)$), any (resp. some) subset of $G$ of size $k$ is a minimal complement. Note that the inclusion 
$$\calA(G) \subseteq \calS(G)$$
holds. 
There are several immediate questions about the structure of these sets, and the common structure of these sets when $G$ runs over a certain class of groups. We explain them below. 

For $* = \cyc$ (resp. $\ab, \nil, \ssol, \sol$), a finite group $G$ is said to be a $*$-group if it is cyclic (resp. abelian, nilpotent, supersolvable, solvable). 
For $* = \emptyset$, a finite group $G$ is said to be a $*$-group if it is satisfies no additional condition other than being a group.

For $*\in \{\cyc, \ab, \nil, \ssol, \sol, \emptyset\}$ and for any positive integer $n$, consider the following subsets of $\{1, 2, \cdots, n\}$ defined as follows.  
\begin{align*}
\calA_n^*
& := \bigcap_{G \text{ is a $*$-group of order }n}\calA(G) ,\\
\calS_n^*
& := \bigcap_{G \text{ is a $*$-group of order }n}\calS(G).
\end{align*}
The sets $\calA_n^\emptyset, \calS_n^\emptyset$ are also denoted by $\calA_n, \calS_n$ respectively. 
In \cite[Question 1]{CoMin1}, the authors asked to determine the structures of the sets $\calA_n^\cyc, \calS_n^\cyc, \calA_n^\ab, \calS_n^\ab, \calA_n, \calS_n$. 
Very recently, some parts of this question have been answered by some of the results of Alon, Kravitz and Larson. They established that the sizes of the minimal complements in the group $\bbZ/n\bbZ$ are exactly $1, 2, \cdots, \lfloor 2n/3\rfloor, n$ \cite[p. 5]{AlonKravitzLarson}. This shows that 
\begin{equation}
\label{Eqn:SnCyc}
\calS^\cyc_n
=
\left\{1, 2, \cdots, \left\lfloor 2n/3\right\rfloor, n\right\}
\end{equation}
for $n\geq 2$.

It would be interesting to investigate the structure of the sets $\calX_n^*$, and the asymptotic behaviour of the sets $\frac 1n \calX_n^*$ for $\calX\in \{\calS, \calA\}, *\in \{\cyc, \ab, \nil, \ssol, \sol, \emptyset\}$. Using the results of Section \ref{Sec:BackLit}, one can conclude several results, as we describe below. 
For $*\in \{\cyc, \ab, \nil, \ssol, \sol, \emptyset\}$, one has the inclusion
$$\calA_n^* \subseteq \calS_n^*,$$
and the inclusions
\begin{equation}
\label{Eqn:Inclu*}
\calX_n 
\subseteq \calX_n^\sol
\subseteq \calX_n^\ssol
\subseteq \calX_n^\nil 
\subseteq \calX_n^\ab 
\subseteq \calX_n^\cyc
\end{equation}
hold for $\calX = \calA, \calS$. 
Moreover, for $*\in \{\cyc, \ab, \nil\}$, 
\begin{equation}
\label{Eqn:Inclumn}
m \calS_n^* \subseteq \calS_{mn}^*
\end{equation}
holds for any positive integers $n,m$ with $\gcd (m,n)=1$, 
and 
\begin{equation}
\label{Eqn:InclumnDiv}
\frac 1n \calS_n^* \subseteq \frac 1m \calS_m^*
\end{equation}
holds for any positive integers $n,m$ with $n\mid m$ and $\gcd(n, m/n) =1$ (see Lemma \ref{Lemma:SnInclusionSm}). 

\begin{lemma}
\label{Lemma:SnInclusionSm}
Equations \eqref{Eqn:Inclumn}, \eqref{Eqn:InclumnDiv} hold. 
\end{lemma}

\begin{proof}
Fix $*\in \{\cyc, \ab, \nil\}$, and let $k$ be an element of $\calS_n^*$. 
Let $G$ be a $*$-group of order $mn$ where $m$ is a positive integer with $\gcd(m,n)=1$. Since $*\in \{\cyc, \ab, \nil\}$ and $\gcd(m, n) = 1$, it follows that $G$ is isomorphic to $G_1 \times G_2$ where $G_1$ (resp. $G_2$) is a group of order $m$ (resp. $n)$. 
Note that $G_2$ contains a subset $A$ of size $k$, which is a minimal complement in $G_2$. Then the subset $G_1 \times A$ of $G_1 \times G_2$ contains $mk$ elements. By Proposition \ref{Prop:CoMinCartesian}, $G_1 \times A$ is a minimal complement to some subset of $G_1\times G_2$. Hence $mk$ is an element of $\calS(G_1\times G_2)$. Thus $mk$ lies in $\calS(G)$ for any $*$-group $G$ of order $mn$. So $mk$ lies in $\calS_{mn}^*$. This establishes Equation \eqref{Eqn:Inclumn}. 
Equation \eqref{Eqn:InclumnDiv} follows from Equation \eqref{Eqn:Inclumn}. 
\end{proof}

For $\calX\in \{\calS, \calA\}, *\in \{\cyc, \ab, \nil, \ssol, \sol, \emptyset\}$, it follows from Theorem \ref{Thm:SmallWorks} that any large finite group contains many minimal complements and thus 
$$\lim_{n\to \infty} |\calX_n^*| = \infty.$$

\begin{questionIntro}
\label{Qn:XnGetsLarger}
For $\calX\in \{\calS, \calA\}$ and $*\in \{\cyc, \ab, \nil, \ssol, \sol, \emptyset\}$, describe the asymptotic property of the sequence $|\calX_n^*|$. 
\end{questionIntro}

For $\calX\in \{\calS, \calA\}, *\in \{\cyc, \ab, \nil, \ssol, \sol, \emptyset\}$, it follows from Theorem \ref{Thm:SmallWorks} that the smallest positive integer lying outside $\calX_n^*$ diverges to $\infty$, i.e., 
$$\lim_{n\to \infty} \min (\{1, 2, \cdots, n\} \setminus \calX_n^*)= \infty.$$

\begin{questionIntro}
\label{Qn:LargeStringAt1}
For $\calX\in \{\calS, \calA\}$ and $*\in \{\cyc, \ab, \nil, \ssol, \sol, \emptyset\}$, describe the asymptotic property of the sequence 
$\min (\{1, 2, \cdots, n\} \setminus \calX_n^*).$
\end{questionIntro}

For $\calX = \calA$, $* = \ab$, it follows from \cite[Corollary 18]{AlonKravitzLarson} that 
$$
\liminf_{n\to \infty} 
\frac{\min (\{1, 2, \cdots, n\} \setminus \calX_n^*)}{\sqrt n} 
\leq 
\sqrt 2,$$
and from \cite[Theorem 3]{AlonKravitzLarson} that 
$$
\min (\{1, 2, \cdots, n\} \setminus \calX_n^*)
= 
O(n^{3/4 + \varepsilon})$$
for any $\varepsilon > 0$. 
Moreover, Alon, Kravitz and Larson conjectured that 
$$
\min (\{1, 2, \cdots, n\} \setminus \calX_n^*)
= 
\widetilde \Theta(\sqrt n)$$
for $\calX = \calA, * = \ab$ \cite[Conjecture 7]{AlonKravitzLarson}.  

From Propositions \ref{Prop:23rdBdd}, \ref{Prop:fini}, it follows that for $\calX\in \{\calS, \calA\}, *\in \{\cyc, \ab, \nil, \ssol, \sol, \emptyset\}$,
$$\lim_{n\to \infty} |(\{1, 2, \cdots, n\} \setminus \calX_n^*)| = \infty.$$

\begin{questionIntro}
\label{Qn:XnComplementGetsLarger}
For $\calX\in \{\calS, \calA\}, *\in \{\cyc, \ab, \nil, \ssol, \sol, \emptyset\}$, determine the asymptotic property of the sequence 
$$ |(\{1, 2, \cdots, \lfloor 2n/3\rfloor \} \setminus \calX_n^*)|.$$
\end{questionIntro}

From Propositions \ref{Prop:23rdBdd}, \ref{Prop:fini}, it follows that for $\calX\in \{\calS, \calA\}, *\in \{\cyc, \ab, \nil, \ssol, \sol, \emptyset\}$, the maximum of $\calX_n^*$ (excluding $n$) is $\leq \frac 23n$ (for $n\geq 2$). Thus for $\calX\in \{\calS, \calA\}, *\in \{\cyc, \ab, \nil, \ssol, \sol, \emptyset\}$,
$$\left(\frac 23, 1\right) \cap \frac 1n \calX_n^* = \emptyset.$$
Moreover, it follows that for $\calX\in \{\calS, \calA\}, *\in \{\cyc, \ab, \nil, \ssol, \sol, \emptyset\}$ and for any $\varepsilon > 0$, 
$$(0, \varepsilon) \cap \frac 1n \calX_n^* \neq \emptyset$$
for large enough $n$ (since $1$ lies in $\calX_n^*$). 
This motivates the following question about the asymptotic behaviour of 
$$\frac 1n \calX^*_n$$
as $n\to \infty$ for $\calX\in \{\calS, \calA\}, *\in \{\cyc, \ab, \nil, \ssol, \sol, \emptyset\}$, 
and the asymptotic behaviour of these sets when $n$ ranges over an infinite set of positive integers (for instance, the set of primes, or the set of all prime powers, or the set of powers of a fixed prime, or the set of square-free integers etc.).

\begin{questionIntro}
\label{Qn:Measure}
Let $\calX\in \{\calS, \calA\}, *\in \{\cyc, \ab, \nil, \ssol, \sol, \emptyset\}$. 
\begin{enumerate}[(i)]
\item 
Let $0 \leq a < b \leq \frac 23$. Evaluate 
$$
\limsup \frac{[a, b] \cap \frac 1n \calX_n^*}
{|\calX_n^*|}, 
\liminf \frac{[a, b] \cap \frac 1n \calX_n^*}
{|\calX_n^*|}.
$$
Does the sequence 
$$
\frac{[a, b] \cap \frac 1n \calX_n^*}
{|\calX_n^*|}
$$
converge? Otherwise, what are its subsequential limits?

\item 
Does there exist a probability measure $\mu$ on $[0, \frac 23]$ such that 
$$
\lim_{n\to \infty} 
\frac{[a, b] \cap \frac 1n \calX_n^*}
{|\calX_n^*|}
= \int_{[0, \frac 23]} \chi_{[a, b]} d\mu
$$
for any $0 \leq a < b \leq \frac 23$, where $\chi_A$ denotes the characteristic function of $A$ for $A\subseteq [0, \frac 23]$?
\end{enumerate}
\end{questionIntro}

\begin{questionIntro}
\label{Qn:BddMinCompDetailed}
Let $\calX\in \{\calS, \calA\}, *\in \{\cyc, \ab, \nil, \ssol, \sol, \emptyset\}$. 
\begin{enumerate}[(i)]

\item 
Determine the open subsets of $[0, \frac 23]$ which do not intersect with $\frac 1n \calX_n^*$ for any/large enough/infinitely many $n$. 

\item 
Determine the open subsets of $[0, \frac 23]$ which have nonempty intersection with $\frac 1n \calX_n^*$ for any/large enough/infinitely many $n$.

\end{enumerate}

\end{questionIntro}

\begin{remark}
In the above questions, one can restrict the integer $n$ from the set of positive integers to some smaller sets, for instance, the set of primes, or the set of all prime powers, or the set of powers of a fixed prime, or the set of square-free integers etc., and study these questions when $n$ varies over such a smaller subset. 
\end{remark}

Using results from Section \ref{Sec:BackLit}, we partially answer Question \ref{Qn:XnComplementGetsLarger} in Proposition \ref{Prop:Qn3Ans}.

\begin{lemma}
\label{Lemma:B1}
Let $n$ be a positive integer and $d_1, \cdots, d_k$ be distinct divisors of $n$ satisfying $d_1\mid d_2 \mid \cdots \mid d_k$. 
For $1\leq i \leq k$, let $B_i$ denote the set defined by 
$$
B_i
:=
\left\{
\frac n{d_i} -1 , \frac n{d_i} - 2, \cdots, \frac n{d_i} - 
\left(\left\lceil \frac n{d_i(2d_i + 1)} \right\rceil  -1\right)
\right\}
.
$$
The union 
$
\cup_{i=1}^k B_i 
$
consists of 
$$\sum_{i=1}^k 
\left(\left\lceil \frac n{d_i(2d_i + 1)} \right\rceil  -1\right)
$$
elements. 
\end{lemma}

\begin{proof}
Note that the sets $B_1, \cdots, B_k$ are pairwise disjoint since for any $i< j$, 
\begin{align*}
\frac n{d_i} - \left(\left\lceil \frac n{d_i(2d_i + 1)} \right\rceil  -1\right)
& \geq 
\frac n{d_i} - \frac n{d_i(2d_i + 1)} \\
& = \frac {2nd_i}{d_i(2d_i + 1)} \\
& = \frac {2n}{2d_i + 1} \\
& > \frac n{2d_i} \\
& \geq 
\frac n{d_j} 
\end{align*}
holds. 
This proves the Lemma. 
\end{proof}

\begin{lemma}
\label{Lemma:B2}
Let $G$ be a group of order $n$. Let $d_1, \cdots, d_k$ be distinct divisors of $n$ such that $d_1\mid \cdots \mid d_k$. Assume that $G$ contains a subgroup of size $n/d_i$ for any $1\leq i \leq k$. 
Then the set $\{1, 2, \cdots, n\} \setminus \calA(G) $ contains the set 
$$
\cup_{i=1}^k B_i ,
$$
and hence contains at least 
$$\sum_{i=1}^k 
\left(\left\lceil \frac n{d_i(2d_i + 1)} \right\rceil  -1\right)
$$
elements. 
\end{lemma}

\begin{proof}

Note that for any integer $r$ satisfying 
$$
1\leq r \leq 
\left(\left\lceil \frac n{d_i(2d_i + 1)} \right\rceil  -1\right),
$$
the inequality 
$$
\frac{\frac n{d_i} -r}{r} 
> 2d_i $$
holds, 
and hence by Proposition \ref{Prop:fini}, $G$ contains a subset of size $n/d_i -r$ which is not a minimal complement. 
Thus for any $1\leq i \leq k$, the set $B_i$ does not intersect with $\calA(G)$. So the set $\{1, 2, \cdots, n\} \setminus \calA(G) $ contains
$$
\cup_{i=1}^k B_i . 
$$
Its cardinality is given by Lemma \ref{Lemma:B1}. 
\end{proof}

\begin{lemma}
\label{Lemma:Bdd}
Let $n$ be a positive integer, $p$ be a prime number and $M$ be a positive integer such that $p^M$ divides $n$. 
Then there exists a sequence of $k$ distinct divisors $d_1, \cdots, d_k$ of $n$ satisfying 
$1< d_1 \mid d_2 \mid \cdots \mid d_k$ 
such that 
$$\sum_{i=1}^k 
\left(\left\lceil \frac n{d_i(2d_i + 1)} \right\rceil  -1\right)
\geq 
\frac n{p^2(2 + \frac 1{p})} - M.
$$

\end{lemma}

\begin{proof}
Let $k = M$ and $d_i = p ^{i}$ for $1\leq i\leq M$. 
Note that 
\begin{align*}
\sum_{i=1}^k 
\left(\left\lceil \frac n{d_i(2d_i + 1)} \right\rceil  -1\right)
& \geq 
\sum_{i=1}^M 
\left(\frac n{d_i(2d_i + 1)}   -1\right)
\\
& \geq 
\sum_{i=1}^M 
\frac n{(2 + \frac 1{p})d_i^2}  -  M
\\
& =  
\frac n{2 + \frac 1{p}} 
\sum_{i=1}^M 
\frac 1{d_i^2}  -  M
\\
& =  
\frac n{2 + \frac 1{p}} \frac 1{p^2} \frac {1 - \frac 1{p^{2M}}}{1 -\frac 1{p^2}} -  M
\\
& 
\geq 
\frac n{p^2(2 + \frac 1{p})}- M .
\end{align*}
\end{proof}

\begin{proposition}
\label{Prop:Qn3Ans}
Let $\{n_k\}_{k\geq 1}$ be a sequence of positive integers. Assume that no term of this sequence gets repeated infinitely often, i.e., it does not admits any constant subsequence. Let $\calP$ be a finite set of primes such that all the prime divisors of any term of this sequence lie in this set. 
Then for any $*\in \{\cyc, \ab, \nil, \ssol, \sol, \emptyset\}$, 
$$|(\{1, 2, \cdots, \lfloor 2n_k/3\rfloor \} \setminus \calA_{n_k}^*)| 
\geq 
\frac {n_k} {(\max \calP)^2(2 + \frac 1{\min \calP})} - \frac{\log n_k}{\log \min \calP}
$$
holds for large enough $k$, and consequently, 
$$\lim_{k\to \infty} |(\{1, 2, \cdots, \lfloor 2n_k/3\rfloor \} \setminus \calA_{n_k}^*)|  = \infty.$$
\end{proposition}

\begin{proof}
By Equation \eqref{Eqn:Inclu*}, it suffices to prove the above inequality for $* = \cyc$. 
Since the terms of the sequence $\{n_k\}_{k\geq 1}$ have prime factors from a finite set of primes and no term of this sequence gets repeated infinitely often, it follows that for any $M> 0$, there exists a positive integer $K$ such that for each $k\geq K$, 
the integer $n_k$ is divisible by $p_k^M$ for some $p_k\in\calP$.  
By Lemmas \ref{Lemma:B1}, \ref{Lemma:B2}, \ref{Lemma:Bdd}, 
for any $k\geq K$, it follows that 
there exists a subset $B$ of $\{1, 2, \cdots, \lfloor 2n_k/3\rfloor \}$ containing at least 
\begin{align*}
\frac {n_k} {p_k^2(2 + \frac 1{p_k})} - M
& = 
\frac {n_k} {p_k^2(2 + \frac 1{p_k})} - \frac{\log n_k}{\log p_k}\\
& \geq 
\frac {n_k} {(\max \calP)^2(2 + \frac 1{\min \calP})} - \frac{\log n_k}{\log \min \calP}
\end{align*}
many elements such that $\calA(G)$ avoids $B$ for any nilpotent group $G$ of order $n_k$. This establishes the result.
\end{proof}

Using Equation \eqref{Eqn:SnCyc}, we partially answer Question \ref{Qn:Measure} in Proposition \ref{Prop:Qn4Ans}. 

\begin{proposition}
\label{Prop:Qn4Ans}
For $\calX = \calS, * = \cyc$, 
$$
\lim_{n\to \infty} 
\frac{(a, b) \cap \frac 1n \calX_n^*}
{|\calX_n^*|} = \frac 32( b-a),
$$
holds for $0 \leq a < b \leq \frac 23$, and Question \ref{Qn:Measure}(ii) admits an answer in the affirmative. 
\end{proposition}

\begin{proof}
From Equation \eqref{Eqn:SnCyc}, it follows that 
$$
\lim_{n\to \infty} 
\frac{(a, b) \cap \frac 1n \calX_n^*}
{|\calX_n^*|} = \frac 32( b-a).
$$
Note that for the probability measure $\mu$ corresponding to the Lebesgue measure on $[0, \frac 23]$, it follows that 
$$
\lim_{n\to \infty} 
\frac{(a, b) \cap \frac 1n \calX_n^*}
{|\calX_n^*|}
= \int_{[0, \frac 23]} \chi_{[a, b]} d\mu
$$
for any $0 \leq a < b \leq \frac 23$.
This answers part (ii). 
\end{proof}

Using Proposition \ref{Prop:fini}, we prove the following lemma, and then establish Proposition \ref{Prop:OpenAvoidance} which partially answers part (i) of Question \ref{Qn:BddMinCompDetailed}. 
\begin{lemma}
\label{Lemma:OpenAvoidance}
Let $G$ be a group of order $n$. Let $i$ be an integer such that $G$ admits a subgroup of  index $i$. 
Then 
$$\left(\frac 2{2i+1} , \frac 1i\right) \cap \frac 1{|G|} \calA(G) = \emptyset.$$
\end{lemma}

\begin{proof}
Let $H$ be a subgroup of $G$ of index $i$. 
For any integer $m$ satisfying
$$\frac {2i}{2i+1} |H| < m < |H|,$$
it follows from Proposition \ref{Prop:fini} that no subset of $H$ containing $m$ elements is a minimal complement in $G$. Thus 
$$\left(\frac {2i}{2i+1} |H|, |H|\right) \cap \calA(G) = \emptyset,$$ 
which yields the result. 
\end{proof}

\begin{proposition}
\label{Prop:OpenAvoidance}
For $*\in \{\cyc, \ab, \nil, \ssol, \sol, \emptyset\}$, 
$$
\bigcup_{i=0}^\infty \left(\frac 2{2p^i+1} , \frac 1{p^i}\right) 
$$ 
does not intersect with $\calA_{p^n}^*$ for any $n\geq 0$. 

\end{proposition}

\begin{proof}
For any group $G$ of order $p^n$, Lemma \ref{Lemma:OpenAvoidance} implies that the set 
$$
\bigcup_{i=0}^n \left(\frac 2{2p^i+1} , \frac 1{p^i}\right) 
$$ 
does not intersect with 
$\frac 1{|G|} \calA(G)$. Since the smallest element of $\frac 1{|G|} \calA(G)$ is $\frac 1{|G|}$, it follows that 
$$
\bigcup_{i=0}^\infty \left(\frac 2{2p^i+1} , \frac 1{p^i}\right) 
$$ 
does not intersect with 
$\frac 1{|G|} \calA(G)$. 
Hence 
$$
\bigcup_{i=0}^\infty \left(\frac 2{2p^i+1} , \frac 1{p^i}\right) 
$$ 
does not intersect with none of 
$\calA_{p^n}^*$
for any $*\in \{\cyc, \ab, \nil, \ssol, \sol, \emptyset\}$. 
\end{proof}

Note that it follows from Equation \eqref{Eqn:SnCyc} that given any nonempty open subset $U$ of $[0, \frac 23]$, it has nonempty intersection with $\frac 1n \calS_n^\cyc$ for large enough $n$. 
This partially answers part (ii) of Question \ref{Qn:BddMinCompDetailed}.

\section{Asymptotic behaviour of co-minimal pairs}
\label{Sec:AsyCoMin}

\begin{definition}\label{Def:CoMin}
A pair $(A, B)$ of two nonempty subsets $A, B$ of a group $G$ is called a co-minimal pair if $A \cdot B = G$, and $A' \cdot B \subsetneq G$ for any $\emptyset \neq A' \subsetneq A$ and $A\cdot B' \subsetneq G$ for any $\emptyset \neq B' \subsetneq B$. 
\end{definition}

For any finite group $G$, let 
$\calS_2(G)$
denote the 
set of pairs of the form $(a, b)$ such that there is a co-minimal pair $(A, B)$ in $G$ with $|A| = a, |B| = b$. 
For $*\in \{\cyc, \ab, \nil, \ssol, \sol, \emptyset\}$ and for any positive integer $n$, consider the following subset of $\{1, 2, \cdots, n\}\times \{1, 2, \cdots, n\}$ defined as follows.  
\begin{align*}
\calS_{2, n}^*
& := \bigcap_{G \text{ is a $*$-group of order }n}\calS_2(G).
\end{align*}
The set $\calS_{2,n}^\emptyset$ is also denoted by $\calS_{2,n}$. 
By Theorem \ref{Thm:SmallWorks}, it follows that 
$$\lim_{n\to \infty} |\calS_{2,n}^*| =  \infty$$
for $*\in \{\cyc, \ab, \nil, \ssol, \sol, \emptyset\}$. 

\begin{questionIntro}
\label{Qn:XnGetsLarger2}
For $*\in \{\cyc, \ab, \nil, \ssol, \sol, \emptyset\}$, describe the asymptotic property of the sequence $|\calS_{2,n}^*|$. 
\end{questionIntro}

Let $I$ denote the unit interval $[0,1]$ and $I^2$ denote the unit square $[0,1]\times [0,1]$. The square $[0, 1/2]\times [0, 1/2]$ is denoted by $I^2_{1/2}$. 
It follows from Proposition \ref{Prop:23rdBdd} that for a co-minimal pair $(A, B)$ in any finite group $G$, 
$$
2|A||B|  - |A| \leq |G| |B|
$$
holds. 
So, for any $*\in \{\cyc, \ab, \nil, \ssol, \sol, \emptyset\}$ and any $(x, y)\in \frac 1n \calS_{2,n}^*$, 
$$\left(x - \frac 1{2n} \right)
\left(y - \frac 12\right) \leq \frac 1 {4n}$$
and 
$$\left(x - \frac 1{2} \right)
\left(y - \frac 1{2n}\right) \leq \frac 1 {4n}$$
hold. For $n\geq 1$, define 
$$U_n 
= 
\left\{(x, y) \,|\, 0 \leq x \leq 1 , 0 \leq y \leq 1, 
2xy 
\leq x + \frac 1n y, 
2xy 
\leq y + \frac 1n x\right\}.$$
Note that $U_n$ contains $\frac 1n \calS_{2,n}^*$ for any $n\geq 1$ and any $*\in \{\cyc,\ab, \nil, \ssol, \sol, \emptyset\}$.

\begin{lemma}
\label{Lemma:Avoidance}
For any $0 < \varepsilon < 1/2$, let $R_\varepsilon$ denote the subset of $I^2$ defined by 
$$
R_\varepsilon = 
([\varepsilon, 1] \times [1/2 + \varepsilon, 1] )
\cup 
([1/2 + \varepsilon, 1] \times [\varepsilon, 1] ).
$$
The region $R_\varepsilon$ does not intersect with $U_n$ for large enough $n$. 
\end{lemma}

\begin{proof}
Let $N$ be a positive integer such that $1/2N + 1/2\sqrt N < \varepsilon$. 
Let $(x, y)$ be an element of $R_\varepsilon$. 
If $(x, y)$ lies in $[\varepsilon, 1] \times [1/2 + \varepsilon, 1]$, then for any $n\geq N$, 
\begin{align*}
\left(x - \frac 1{2n} \right)\left(y - \frac 12 \right) 
& \geq 
\varepsilon 
\left(x - \frac 1{2n} \right)\\
& \geq 
\varepsilon 
\left(x - \frac 1{2N} \right)\\
& > 
\frac 1{2\sqrt N} \cdot \frac 1{2\sqrt N} \\
& \geq  \frac 1 {4n}.
\end{align*}
If $(x, y)$ lies in $[1/2 + \varepsilon, 1] \times [\varepsilon, 1]$, then for any $n\geq N$, 
\begin{align*}
\left(x - \frac 12 \right)\left(y - \frac 1{2n}\right) 
& \geq 
\varepsilon 
\left(y - \frac 1{2n} \right)\\
& \geq 
\varepsilon 
\left(y - \frac 1{2N} \right)\\
& > 
\frac 1{2\sqrt N} \cdot \frac 1{2\sqrt N} \\
& \geq  \frac 1 {4n}.
\end{align*}
So no element of $R_\varepsilon$ lies in $U_n$ for $n\geq N$. 
\end{proof}

Note that for any $\varepsilon >0$ and $*\in \{\cyc,\ab, \nil, \ssol, \sol, \emptyset\}$, the set $R_\varepsilon$ avoids  
$\frac 1n \calS_{2,n}^*$ for large enough $n$.

\begin{questionIntro}
\label{Qn:XnComplementGetsLarger2}
For $*\in \{\cyc, \ab, \nil, \ssol, \sol, \emptyset\}$ and for $0 < \varepsilon < 1/2$, determine the asymptotic property of the sequence 
$$ |(([1/n, 2/n, \cdots, n/n]^2\setminus R_\varepsilon)\setminus \calS_{2,n}^*)|.$$
\end{questionIntro}

Note that 
$$|(([1/n, 2/n, \cdots, n/n]^2\setminus R_\varepsilon)\setminus \calS_{2,n}^*)|
\geq 
\left\lfloor \frac n2  \right\rfloor^2 -1$$
since $(1/n, i/n), (i/n, 1/n)$ belongs to $([1/n, 2/n, \cdots, n/n]^2\setminus R_\varepsilon)\setminus \calS_{2,n}^*$ for $1\leq i \leq \lfloor n/2\rfloor$. Thus a more interesting question would be to study the asymptotic property of the sequence 
$$ |(([1/n, 2/n, \cdots, n/n]^2\setminus (R_\varepsilon \cup R'_n))\setminus \calS_{2,n}^*)|$$
where $R'_n$ consists of those points of $[1/n, 2/n, \cdots, n/n]^2$ which are avoided by $\calS_{2,n}^*$ for ``obvious reasons''. Thus, $R'_n$ contains those points satisfying $xy < 1/n$.  

\begin{questionIntro}
\label{Qn:Measure2}
Let $*\in \{\cyc, \ab, \nil, \ssol, \sol, \emptyset\}$. 
\begin{enumerate}[(i)]
\item 
Let $0 \leq a < b \leq 1, 0 \leq c < d \leq 1$. Evaluate 
$$
\limsup \frac{([a, b]\times [c,d]) \cap \frac 1n \calS_{2,n}^*}
{|\calS_{2,n}^*|}, 
\liminf \frac{([a, b]\times [c,d]) \cap \frac 1n \calS_{2,n}^*}
{|\calS_{2,n}^*|}.
$$
Does the sequence 
$$
\frac{([a, b]\times [c,d]) \cap \frac 1n \calS_{2,n}^*}
{|\calS_{2,n}^*|}
$$
converge? Otherwise, what are its subsequential limits?

\item 
Does there exist a probability measure $\mu$ on $I^2$ such that 
$$
\lim_{n\to \infty} 
\frac{([a, b] \times [c,d])\cap \frac 1n \calS_{2,n}^*}
{|\calS_{2,n}^*|}
= \int_{I^2} \chi_{[a, b]\times [c,d]} d\mu
$$
for any $0 \leq a < b \leq 1, 0 \leq c < d \leq 1$, where $\chi_A$ denotes the characteristic function of $A$ for $A\subseteq I^2$?
\end{enumerate}
\end{questionIntro}

\begin{questionIntro}
\label{Qn:BddMinCompDetailed2}
Let $*\in \{\cyc, \ab, \nil, \ssol, \sol, \emptyset\}$. 
\begin{enumerate}[(i)]

\item 
Determine the open subsets of $I^2$ which do not intersect with $\frac 1n \calS_{2,n}^*$ for any/large enough/infinitely many $n$. 

\item 
Determine the open subsets of $I^2$ which have nonempty intersection with $\frac 1n \calS_{2,n}^*$ for any/large enough/infinitely many $n$.

\end{enumerate}

\end{questionIntro}

Note that Lemma \ref{Lemma:Avoidance} partially answers part (i) of the above question.

\begin{definition}
A $k$-tuple $(A_1, \cdots, A_k)$ of non-empty subsets of a group $G$ is said to be a co-minimal $k$-tuple if 
$$A_1 \cdot A_2 \cdot \cdots \cdot A_k = G$$
and for any $1\leq i \leq k$ and for any $a_i \in A_i$, 
$$A_1 \cdot A_2 \cdot \cdots \cdot(A_i \setminus \{a_i\}) \cdot \cdots \cdot A_k \neq G.$$
\end{definition}

Note that one can define an analogue of $\calS_{2, n}^*$ as follows. 
Let $k\geq 2$ be an integer. 
For any finite group $G$, let 
$\calS_k(G)$
denote the 
set of pairs of the form $(a_1, \cdots, a_k)$ such that there is a co-minimal $k$-tuple $(A_1, \cdots, A_k)$ in $G$ with $|A_i| = a_i$ for $1\leq i \leq k$. 
For $*\in \{\cyc, \ab, \nil, \ssol, \sol, \emptyset\}$ and for any positive integer $n$, consider the following subset of $\{1, 2, \cdots, n\}^k$ defined as follows.  
\begin{align*}
\calS_{k, n}^*
& := \bigcap_{G \text{ is a $*$-group of order }n}\calS_k(G).
\end{align*}
The set $\calS_{k,n}^\emptyset$ is also denoted by $\calS_{k,n}$. 
By Theorem \ref{Thm:SmallWorks}, it follows that 
$$\lim_{n\to \infty} |\calS_{k,n}^*| =  \infty$$
for $*\in \{\cyc, \ab, \nil, \ssol, \sol, \emptyset\}$. 
One could ask the following questions, which are analogous to Questions \ref{Qn:XnGetsLarger2}, 
\ref{Qn:XnComplementGetsLarger2}, 
\ref{Qn:Measure2}, 
\ref{Qn:BddMinCompDetailed2}.

\begin{questionIntro}
For $*\in \{\cyc, \ab, \nil, \ssol, \sol, \emptyset\}$, describe the asymptotic property of the sequence $|\calS_{k,n}^*|$. 
\end{questionIntro}

\begin{questionIntro}
For $*\in \{\cyc, \ab, \nil, \ssol, \sol, \emptyset\}$ and for $0 < \varepsilon < 1/2$, determine the asymptotic property of the sequence 
$$ |(([1/n, 2/n, \cdots, n/n]^k\setminus R_n)\setminus \calS_{k,n}^*)|.$$
where $R_n$ consists of those points of $[1/n, 2/n, \cdots, n/n]^k$ which are avoided by $\calS_{k,n}^*$ for ``obvious reasons''. For instance, $R'_n$ contains those points satisfying $x_1\cdots x_k < 1/n^{k-1}$. 
\end{questionIntro} 

\begin{questionIntro}
Let $*\in \{\cyc, \ab, \nil, \ssol, \sol, \emptyset\}$. 
\begin{enumerate}[(i)]
\item 
Let $a_1, b_1, \cdots, a_k, b_k$ be real numbers satisfying $0 \leq a_i < b_i \leq 1$ for all $1\leq i \leq k$. Evaluate 
$$
\limsup \frac{([a_1, b_1] \times \cdots \times [a_k,b_k]) \cap \frac 1n \calS_{k,n}^*}
{|\calS_{k,n}^*|}, 
\liminf \frac{([a_1, b_1] \times \cdots \times [a_k,b_k]) \cap \frac 1n \calS_{k,n}^*}
{|\calS_{k,n}^*|}.
$$
Does the sequence 
$$
\frac{([a_1, b_1] \times \cdots \times [a_k,b_k]) \cap \frac 1n \calS_{k,n}^*}
{|\calS_{k,n}^*|}
$$
converge? Otherwise, what are its subsequential limits?

	\item 
Does there exist a probability measure $\mu$ on $I^k$ such that 
$$
\lim_{n\to \infty} 
\frac{([a_1, b_1] \times \cdots \times [a_k,b_k])\cap \frac 1n \calS_{k,n}^*}
{|\calS_{k,n}^*|}
= \int_{I^k} \chi_{[a_1, b_1] \times \cdots \times [a_k,b_k]} d\mu
$$
for any real numbers $a_1, b_1, \cdots, a_k, b_k$ satisfying $0 \leq a_i < b_i \leq 1$ for all $1\leq i \leq k$, where $\chi_A$ denotes the characteristic function of $A$ for $A\subseteq I^k$?
\end{enumerate}
\end{questionIntro}

\begin{questionIntro}
Let $*\in \{\cyc, \ab, \nil, \ssol, \sol, \emptyset\}$. 
\begin{enumerate}[(i)]

\item 
Determine the open subsets of $I^k$ which do not intersect with $\frac 1n \calS_{k,n}^*$ for any/large enough/infinitely many $n$. 

\item 
Determine the open subsets of $I^k$ which have nonempty intersection with $\frac 1n \calS_{k,n}^*$ for any/large enough/infinitely many $n$.

\end{enumerate}

\end{questionIntro}

\section{Acknowledgements}
The first author is supported by the ISF Grant no. 662/15. He wishes to thank the Department of Mathematics at the Technion where a part of the work was carried out. The second author would like to acknowledge the Initiation Grant from the Indian Institute of Science Education and Research Bhopal, and the INSPIRE Faculty Award from the Department of Science and Technology, Government of India.

\def\cprime{$'$} \def\Dbar{\leavevmode\lower.6ex\hbox to 0pt{\hskip-.23ex
  \accent"16\hss}D} \def\cfac#1{\ifmmode\setbox7\hbox{$\accent"5E#1$}\else
  \setbox7\hbox{\accent"5E#1}\penalty 10000\relax\fi\raise 1\ht7
  \hbox{\lower1.15ex\hbox to 1\wd7{\hss\accent"13\hss}}\penalty 10000
  \hskip-1\wd7\penalty 10000\box7}
  \def\cftil#1{\ifmmode\setbox7\hbox{$\accent"5E#1$}\else
  \setbox7\hbox{\accent"5E#1}\penalty 10000\relax\fi\raise 1\ht7
  \hbox{\lower1.15ex\hbox to 1\wd7{\hss\accent"7E\hss}}\penalty 10000
  \hskip-1\wd7\penalty 10000\box7}
  \def\polhk#1{\setbox0=\hbox{#1}{\ooalign{\hidewidth
  \lower1.5ex\hbox{`}\hidewidth\crcr\unhbox0}}}
\providecommand{\bysame}{\leavevmode\hbox to3em{\hrulefill}\thinspace}
\providecommand{\MR}{\relax\ifhmode\unskip\space\fi MR }
\providecommand{\MRhref}[2]{%
  \href{http://www.ams.org/mathscinet-getitem?mr=#1}{#2}
}
\providecommand{\href}[2]{#2}

\end{document}